\documentclass[10pt,a4paper,twoside]{article}
\usepackage{amssymb}
\usepackage{amsmath}
\usepackage{fancyhdr}
\usepackage{amsthm}
\usepackage{rotate}
\usepackage{graphicx}
\usepackage[T1]{fontenc}
\usepackage{afterpage}  %
\usepackage{color}

\newtheorem{theorem}{Theorem}[section]
\newtheorem{proposition}[theorem]{Proposition}

\newtheorem{corollary}[theorem]{Corollary}

\newtheorem{example}[theorem]{Example}

\begin{document}

\afterpage{\rhead[]{\thepage} \chead[\small W.A. Dudek and R.A.R Monzo]{\small Biquasigroups linear over a group} \lhead[\thepage]{}}                  %%%%%%%%% complete

\begin{center}
\vspace*{2pt}
{\Large \textbf{Biquasigroups linear over a group}}\\[26pt]
{\large \textsf{\emph{Wieslaw A. Dudek \ and \ Robert A. R. Monzo}}}  
\\[26pt]
\end{center}
\textbf{Abstract.} {\footnotesize We determine the structure of biquasigroups $(Q,\circ,*)$ satisfying variations of Polonijo’s Ward double quasigroup identity $(x\circ z)*(y\circ z)=x*y$, including those that are linear over a  group.           
}

%%%%%%%%%%%%%%%%%%%%%%%%%%%%%%%%%%%%%%%%%%%%%%%%%%%%%%%%%%%%%%%%%%%%%%%%%%%%%%%%%%%%%%%%%

\section*{\centerline{1. Introduction}}\setcounter{section}{1}
J.M. Cardoso and C.P. da Silva, inspired by Ward's paper \cite{Ward} on postulating the inverse operations in groups, introduced in \cite{Car} the notion of {\em Ward quasigroups} as  quasigroups $(Q,\circ)$ containing an element $e$ such that $x\circ x=e$ for all $x\in Q$, and satisfying the identity $(x\circ y)\circ z=x\circ(z\circ(e\circ y))$. Polonijo \cite{Pol2} proved that these two conditions can be replaced by the identity:
\begin{eqnarray}\label{e1}
&&(x\circ z)\circ(y\circ z)=x\circ y.
\end{eqnarray}  

In \cite{Car} it is proved that if $(Q,\circ)$ is a Ward quasigroup, then $(Q,\cdot)$, where $x\cdot y=x\circ (e\circ y)$, is a group in which $e=x\circ x$ and $x^{-1}=e\circ x$ for all $x\in Q$. Also, $x\circ e=x$, $e\circ (e\circ x)=x$ and $e\circ (x\circ y)=y\circ x$. 
Conversely, if $(Q,\cdot)$ is a group, then $Q$ with the operation $x\circ y=x\cdot y^{-1}$ is a Ward quasigroup (cf.\ \cite{Ward}). Other characterizations of Ward quasigroups can be found in \cite{Ch} and \cite{Sil}, some applications in \cite{JV}. Note that the Ward quasigroups corresponding to commutative groups sometimes are called {\em subtractive quasigroups} (cf.\ \cite{Mor} and \cite{Whi}). A Ward quasigroup $(Q,\circ)$ is subtractive if and only if it is medial (that is, it satisfies the identity $(x\circ y)\circ(z\circ w)=(x\circ z)\circ (y\circ w))$ if and only if it is left modular (that is, it satisfies the identity $x\circ (y\circ z)=z\circ (y\circ x)$) (cf.\ Lemma 2.4, \cite{DM1}).

A {\em biquasigroup}, i.e. an algebra of the form $(Q,\circ,*)$ where $(Q,\circ)$ and $(Q,*)$ are quasigroups, is called a {\em Ward double quasigroup} if it satisfies the identity
\begin{eqnarray}\label{e2}
&&(x\circ z)*(y\circ z)=x*y.
\end{eqnarray}  

Obviously each Ward quasigroup $(Q,\circ)$ can be considered as a Ward double quasigroup of the form $(Q,\circ,\circ)$.
Ward double quasigroups have a similar characterization as Ward quasigroups. 

\begin{theorem}\label{T11} {\rm (cf.\ \cite{Pol1})} A biquasigroup $(Q,\circ,*)$ is a Ward double quasigroup if and only if  there is a group $(Q,+)$ and bijections $\alpha,\beta$ on $Q$ such that $x\circ y=x -\beta y$ and $x*y=\alpha(x-y)$. 
\end{theorem}

Note that Ward double quasigroups are distinct from the double Ward quasigroups considered by Fiala (cf.\ \cite{Fia}).

\medskip
Let us consider the identity \eqref{e2}. Keeping the variables $x,y$ and $z$ the same and varying only the quasigroup operations $\circ$ and $*$, there are 
sixteen possible identities. Eight of these have reversible versions obtained by replacing the operation $\circ$  with the operation $*$ and, simultaneously, replacing the operation $*$ with the operation $\circ$.

For example, the identity $(x\circ z)*(y*z)=x\circ y$ has the reversible version $(x*z)\circ (y\circ z)=x*y$. So, if we are to consider all possible versions of Theorem \ref{T11}, we need to explore the following identities:
\begin{eqnarray}
\label{e3}
&&(x\circ z)\circ (y\circ z)=x\circ y,\\
\label{e4}
&&(x\circ z)\circ (y\circ z)=x*y,\\
\label{e5}
&&(x\circ z)\circ(y* z)=x\circ y,\\
\label{e6}
&&(x\circ z)*(y\circ z)=x\circ y,\\
\label{e7}
&&(x\circ z)\circ (y* z)=x* y,\\
\label{e8}
&&(x\circ z)* (y* z)=x\circ y,\\
\label{e9}
&&(x\circ z)* (y* z)=x*y.
\end{eqnarray}  
  
The biquasigroup $(Q,\circ,\circ)$ satisfies identity \eqref{e3} if and only if $(Q,\circ)$ is a Ward quasigroup. 
Our interest is in finding non-trivial models of the other six identities, where `non-trivial' means that the set $Q$ has more than one element. In particular, since Ward quasigroups are unipotent, we will be interested in biquasigroups $(Q,\circ,*)$ where $(Q,\circ)$ or $(Q,*)$ is unipotent, both are unipotent or when one or both are Ward quasigroups.

Note that a biquasigroup $(Q,\circ,*)$, where $(Q,*)$ is a commutative group and $x\circ y=x*y^{-1}$ satisfies identities \eqref{e1} through \eqref{e9} if and only if $(Q,*)$ is a Boolean group.

%%%%%%%%%%%%%%%%%%%%%%%%%%%%%%%%%%%%%%%%%%%%%%%%%%%%%%%%%%%%%%%%%%%%%%%%%%%%%%%%%%%%%%%%%%%%%%%%%%%%%%%%%%%%%%
%%%%%%%%%%%%%%%%%%%%%%%%%%%%%%%%%%%%%%%%%%%%%%%%%%%%%%%%%%%%%%%%%%%%%%%%%%%%%%%%%%%%%%%%%%%%%%%%%%%%%%%%%%%%%%
\section*{\centerline{2. Main Results}}\setcounter{section}{2}\setcounter{theorem}{0}

We will now characterize the biquasigroups satisfying the identities \eqref{e2} to \eqref{e9}.
First we will describe their general properties then we will characterize biquasigroups linear over a group and satisfying identities \eqref{e2} to \eqref{e9}.

\medskip\noindent
{\bf 1.} 
Recall that a quasigroup $(Q,\cdot)$ is {\em linear} over a group (cf. \cite{Scerb}) if there exists a group $(Q,+)$, its automorphisms $\varphi,\psi$ and  $a\in Q$ such that $x\cdot y=\varphi x+a+\psi y$ for all $x,y\in Q$. Consequently, a biquasigroup $(Q,\circ,*)$ will be called {\em linear} over a group if both its quasigroups $(Q,\circ)$ and $(Q,*)$ are linear over the same group, i.e. if there is a group $(Q,+)$, its automorphisms $\varphi,\psi,\alpha,\beta$ and elements $a,b\in Q$ such that
$$
x\circ y=\varphi x+a+\psi y \qquad {\rm and} \qquad x*y=\alpha x+b+\beta y.
$$

According to the Toyoda Theorem (cf.\ \cite{Scerb}), a quasigroup $(Q,\cdot)$ is medial if and only if it is linear over a commutative group with commuting automorphisms $\varphi,\psi$. 
In an analogous way we can shows that a quasigroup $(Q,\cdot)$ is {\em paramedial}, that is it satisfies the identity $(x\cdot y)\cdot(z\cdot u)=(u\cdot y)\cdot (z\cdot x)$ if and only if it is linear over a commutative group with automorphisms $\varphi$, $\psi$ such that $\varphi^2=\psi^2$. Based on these facts we say that a biquasigroup $(Q,\circ,*)$ is {\em medial} ({\em paramedial}) if both its quasigroups $(Q,\circ)$ and $(Q,*)$ are medial (paramedial)  and linear over the same commutative group. 

A biquasigroup $(Q,\circ,*)$ is {\em unipotent} if there is $q\in Q$ such that $x\circ x=q=x*x$ for and all $x\in Q$. If both quasigroups $(Q,\circ)$ and $(Q,*)$ are idempotent then we say that $(Q,\circ,*)$ is an {\em idempotent biquasigroup}. 

%%%%%%%%%%%%%%%%%%%%%%%%%%%%%%%%%%%%%%%%%%%%%%%%%%%%%%%%%%%%%%%%%%%%%%%%%%%
\bigskip\noindent {\bf 2.} 
We will start with biquasigroups satisfying the identity \eqref{e2}. 

\smallskip
A general characterization of such biquasigroups is given by Theorem \ref{T11}. 
Now we describe a biquasigroup linear over a group $(Q,+)$ and satisfying identity \eqref{e2}.

 From \eqref{e2} for $x=y=z=0$ we obtain $\alpha a+b+\beta a=b$, This together with \eqref{e2} implies $\varphi=\varepsilon$ (the identity map). Thus 
$$
\alpha x+\alpha a+\alpha\psi z+b+\beta y+\beta a+\beta\psi z=\alpha x+b+\beta y.
$$ 
This for $z=0$ gives  
$$
\alpha a+b+\beta y +\beta a=b+\beta y=\alpha a+b+\beta a+\beta y.
$$ 
So $\beta y+\beta a=\beta a+\beta y$, i.e. $a$ is in the center $Z(Q,+)$ of the group $(Q,+)$. Thus using \eqref{e2} and the above facts we obtain 
$\alpha\psi z+b+\beta y+\beta\psi z=b+\beta y$. Hence $\alpha v+u+\beta v=u$ for all $u,v\in Q$.
Thus $\beta=-\alpha$ and consequently $\alpha v+u=u+\alpha v$ for all $u,v\in Q$, which means that $(Q,+)$ is a commutative group.

\smallskip
In this way we have proved the “only if” part of the following Theorem. The second part is trivial.

\begin{theorem}\label{T22}
A biquasigroup $(Q,\circ,*)$ linear over a group $(Q,+)$ is a Ward double quasigroup $($that is, it satisfies \eqref{e2}$)$ if and only if $(Q,+)$ is a commutative group, 
$x\circ y=x+\psi y+a$ and $x*y=\alpha x-\alpha y+b$.
\end{theorem}
Obviously such a biquasigroup is medial. The quasigroup $(Q,\circ)$ has a right neutral element and the quasigroup $(Q,*)$ is unipotent. Moreover, a biquasigroup  $(Q,\circ,*)$ satisfying \eqref{e2} is paramedial if and only if $\psi^2=\varepsilon$.

\medskip %%%%%%%%%%%%%%%%%%%%%%%%%%%%%%%%%%%%%%%%%%%%%%%%%%%%%%%%%%%%%%%%%%%%%%%%
\medskip\noindent {\bf 3.}
Now consider biquasigroups satisfying the identity \eqref{e3}. 

\smallskip
Since this identity contains only one operation, it is enough to examine the quasigroup $(Q,\circ)$. 
Quasigroups satisfying \eqref{e3} were characterized at the beginning of this paper. If a quasigroup $(Q,\circ)$ linear over a group $(Q,+)$ satisfies \eqref{e3}, then $\varphi=\varepsilon$ and $a+\psi a=0$.
So \eqref{e3} for $y=0$, can be reduced to $\psi z+\psi^2 z=0$. This means that $\psi z=-z$ and $(Q,+)$ is a commutative group. Consequently $x\circ y =x-y+a$. 
 
\begin{theorem}\label{T23} 
A quasigroup $(Q,\circ)$ linear over a group $(Q,+)$ satisfies \eqref{e3} if and only if $(Q,+)$ is a commutative group and $x\circ y=x-y+a$ for some fixed $a\in Q$.
\end{theorem}

This quasigroup is medial, paramedial, unipotent and has a right neutral element.

Note that $(Q,\circ)$ is a Ward quasigroup if and only if there is a group $(Q,+)$ and an element $a\in Q$ such that $x\circ y=x-y+a$. The group $(Q,+)$  need not be commutative.

%%%%%%%%%%%%%%%%%%%%%%%%%%%%%%%%%%%%%%%%%%%%%%%%%
\bigskip\noindent {\bf 4.}
Will now consider a biquasigroup $(Q,\circ,*)$ satisfying the identity \eqref{e4}, i.e.
\begin{eqnarray*}
&&(x\circ z)\circ (y\circ z)=x*y.
\end{eqnarray*}

\begin{theorem}\label{T24} If a biquasigroup $(Q,\circ,*)$ satisfies the identity $\eqref{e4}$, then both quasigroups $(Q,\circ)$ and $(Q,*)$ are unipotent with $q\in Q$ such that $x\circ x=q=x*x$ and $x*y=(x\circ y)\circ q=q\circ (y\circ x)$ for all $x,y\in Q$. 
\end{theorem}
\begin{proof}  If $(Q,\circ)$ is idempotent, then $x=(x\circ x)\circ(x\circ x)=x*x$. So $(Q,*)$ is idempotent too. If $(Q,*)$ is idempotent, then $x=x*x=(x\circ z)\circ (x\circ z)$ for all $x,z\in Q$. In particular, for $z=x'\in Q$ such that $x=x\circ x'$ we obtain $x=x\circ x$. This shows that both these quasigroups are idempotent or none of them are idempotent.

If both are idempotent, then $x\circ z=(x\circ z)\circ (x\circ z)=x*x=x=x\circ x$ for all $x,z\in Q$, which implies $x=z$. Hence $Q$ has only one element. So it is unipotent.

Now suppose both quasigroups $(Q,\circ)$ and $(Q,*)$ are not idempotent. 
Then there exists $b\in Q$ such that $b*b=q\ne b$ and for any $x\in Q$ there exist $x',x''\in Q$ such that $b\circ x'=x$ and $x\circ x''=x$. Then $x\circ x=(b\circ x')\circ (b\circ x')=b*b=q$ and $x*x=(x\circ x'')\circ(x\circ x'')=x\circ x=q$. Hence, $(Q,\circ)$ and $(Q,*)$ are unipotent, with $q=x\circ x=x*x$ for all $x\in Q$. Also, $x*y=(x\circ x)\circ (y\circ x)=q\circ (y\circ x)$ and $x*y=(x\circ y)\circ (y\circ y)=(x\circ y)\circ q$.   
\end{proof}
\begin{corollary}
If a biquasigroup $(Q,\circ,*)$ satisfies $\eqref{e4}$ and $(Q,\circ)$ has a right neutral element, then $(Q,\circ)=(Q,*)$ is a Ward quasigroup. If $(Q,\circ)$ has a left neutral element, then $x*y=y\circ x$. If $(Q,\circ)$ has a neutral element, then $(Q,\circ)=(Q,*)$ is a commutative Ward quasigroup. 
\end{corollary}

Any medial unipotent quasigroup $(Q,\circ)$ can be 'extended' to a medial unipotent biquasigroup $(Q,\circ,*)$ satisfying the identity \eqref{e4}, as follows.

\begin{proposition}
If $(Q,\circ)$ is a medial unipotent quasigroup, then $(Q,\circ,*)$, where $x\circ x=q$ and $x*y=(x\circ y)\circ q$ for all $x,y\in Q$, is a biquasigroup satisfying $\eqref{e4}$. Moreover, if $q$ is a left neutral element of $(Q,\circ)$, then $x*y=y\circ x$. 
\end{proposition}
\begin{proof}
Indeed, $(Q,*)$ is a quasigroup and $x*y=(x\circ y)\circ q=(x\circ y)\circ (z\circ z)=(x\circ z)\circ (y\circ z)$. Also, if $q$ is a left neutral element of $(Q,\circ)$, then $x*y=(x\circ y)\circ q=(x\circ y)\circ (x\circ x)=(x\circ x)\circ (y\circ x)=y\circ x$.
\end{proof}

Let $(Q,\circ,*)$ be a biquasigroup linear over a group $(Q,+)$. If it satisfies \eqref{e4}, then 
$\varphi a+a+\psi a=b$ and $\alpha=\varphi^2$. So \eqref{e4} for $x=y=0$ and $\psi z=a$ gives 
$2\varphi a+a+2\psi a=b=\varphi a+a+\psi a$ which implies $a=b$. Consequently $a\circ a=a$. Thus, by Theorem \ref{T24}, $a=q$ and $x*y=a\circ (y\circ x)$. Hence
$$
x*0=a\circ (0\circ x)=\alpha x+a=\varphi a+a+\psi a+\psi^2 x=a+\psi^2 x
$$ 
and
$$
a=x\circ x=\alpha x+a+\beta x=a+\psi^2 x+\beta x.
$$
This gives $\psi^2 x+\beta x=0$, i.e. $\beta=-\psi^2$. Hence $x*y=\varphi^2 x+a-\psi^2 y$.
Since $x\circ x=a=z*z$ we also have $\varphi x+a=a-\psi x$ and $\varphi^2 z+a=a+\psi^2 z$. 
This for $x=\varphi z$ gives $\varphi^2 z+a=a-\psi\varphi z$. Hence $a+\psi^2 z=a-\psi\varphi z$. Consequently, $\psi=-\varphi$ and $x\circ y=\varphi x+a-\varphi y$. So $a=x\circ x=\varphi x+a-\varphi x$. Thus $a\in Z(Q,+)$. Also $\varphi^2=\psi^2$.

Therefore, $x\circ y=\varphi x+a-\varphi y$ and $x*y=\varphi^2 x+a-\varphi^2 y$. Inserting these operations to \eqref{e4} we obtain $-\varphi^2 z-\varphi^2 y+\varphi^2 z=-\varphi^2 y$ for all $y,z\in Q$. Hence $(Q,+)$ is a commutative group. Consequently $(Q,\circ,*)$ is medial and unipotent. This proves the “only if” part of the Theorem \ref{T24a} below. The proof of the “if” part follows from a direct calculation and is omitted

\begin{theorem}\label{T24a}
A biquasigroup $(Q,\circ,*)$ linear over a group $(Q,+)$ satisfies the identity \eqref{e4} if and only if $(Q,+)$ is a commutative group, $x\circ y=\varphi x+a-\varphi y$ and $x*y=\varphi^2 x+a-\varphi^2 y$.
\end{theorem}

It is clear that such a biquasigroup is medial and paramedial. If $a=0$ then it is unipotent.

%%%%%%%%%%%%%%%%%%%%%%%%%%%%%%%%%%%%%%%%%%%%%%%%%%%%%%%%%%%%%%%%%%%%%%%%%%%%%%%%%%%%%%%%%%%%%%%
\bigskip\noindent {\bf 5.}
Will now consider a biquasigroup $(Q,\circ,*)$ satisfying the identity \eqref{e5}, i.e.
$$(x\circ z)\circ (y*z)=x\circ y.$$

\begin{theorem}\label{T25} If a biquasigroup $(Q,\circ,*)$ satisfies the identity $\eqref{e5}$, then both quasigroups $(Q,\circ)$ and $(Q,*)$ have only one idempotent. This idempotent is a right neutral element of these quasigroups. Moreover, $(Q,*)$ is unipotent.
\end{theorem}
\begin{proof} For each $x\in Q$ there is uniquely determined $\overline{x}\in Q$ such that $x\circ\overline{x}=x$. Then for $x,y\in Q$, by \eqref{e5}, we have 
$$x\circ y=(x\circ \overline{x})\circ (y*\overline{x})=x\circ(y*\overline{x}).
$$ So, $y=y*\overline{x}$ for each $y\in Q$. Also $y\circ y=(y\circ \overline{x})\circ (y*\overline{x})=(y\circ\overline{x})\circ y$, hence $y=y\circ\overline{x}$ for all $y\in Q$. Thus $e=\overline{x}$ is a right neutral element of $(Q,\circ)$ and $(Q,*)$. There are no other idempotents in $(Q,\circ)$ and $(Q,*)$. Indeed, if $a*a=a$, then for each $x\in Q$ 
$$x\circ a=(x\circ a)\circ (a*a)=(x\circ a)\circ a,$$ so $x\circ a=x=x\circ e$. Hence $a=e$. Similarly for $a\circ a=a$ we have 
$$a\circ x=(a\circ a)\circ (x*a)=a\circ (x*a),$$ which implies $x*a=x=x*e$, so also in this case $a=e$.

For each $x\in Q$ there exists $x'\in Q$ such that $x*x=x\circ x'$. Thus, by \eqref{e5}, 
$$(x\circ x)\circ e=x\circ x=(x\circ x)\circ (x*x)=(x\circ x)\circ (x\circ x'),$$ which implies $e=x\circ x'=x*x$. So, $(Q,*)$ is unipotent.
\end{proof}

The following example shows that $(Q,\circ)$ may not be unipotent.

\begin{example}\rm
Let $(Q,\circ)$ be a group. Then $(Q,\circ,*)$, where $x*y=y^{-1}\circ x$ is an example of a biquasigroup satisfying \eqref{e5} in which only one of quasigroups $(Q,\circ)$ and $(Q,*)$ has a left neutral element. Moreover, $(Q,*)$ is unipotent but $(Q,\circ)$ is unipotent only in the case when it is a Boolean group.
\end{example}
\begin{corollary}
If in a biquasigroup $(Q,\circ,*)$ satisfying $\eqref{e5}$ one of quasigroups $(Q,\circ)$ or $(Q,*)$ is idempotent, then $Q$ has only one element.
\end{corollary}

\begin{proposition} Let $(Q,\circ,*)$ be a biquasigroup satisfying \eqref{e5}. If $(Q,\circ)$ is a Ward quasigroup, then $(Q,\circ)=(Q,*)$.
\end{proposition}
\begin{proof} Since $(Q,\circ)$ is a Ward quasigroup, there exists a group $(Q,\cdot)$ such that $x\circ y=x\cdot y^{-1}$ and $e\circ (x\circ y)=y\circ x$, where $e$ is the neutral element of the group $(Q,\cdot)$ (cf.\ \cite{Car}). Then $x\circ y=(x\circ x)\circ (y*x)=e\circ (y*x)$ and so $x\circ y=e\circ (y\circ x)= e\circ (e\circ (x*y))=x*y$. Hence $(Q,\circ)=(Q,*)$.
\end{proof}

\begin{proposition} Let $(Q,\circ,*)$ be a biquasigroup satisfying \eqref{e5}. If $(Q,\circ)$ is medial and unipotent, then $(Q,\circ)=(Q,*)$.
\end{proposition}
\begin{proof}
For every $x,y\in Q$ there exists $z\in Q$ such that $x*y=x\circ z$. Since $(Q,\circ)$ is medial, 
$$(x\circ x)\circ e=x\circ x=(x\circ y)\circ (x*y)=(x\circ y)\circ (x\circ z)=(x\circ x)\circ (y\circ z).$$ Thus, $y\circ z=e=y\circ y$, where $e$ is the right neutral element of $(Q,\circ)$. Therefore $y=z$ and consequently, $x*y=x\circ y$.
\end{proof}

\begin{proposition}
A biquasigroup $(Q,\circ,*)$ linear over a group $(Q,+)$ satisfies \eqref{e5} if and only if a group $(Q,+)$ is a commutative group, $x\circ y=x+\psi y-\psi b$ and $x*y=x-y+b$.
\end{proposition}
\begin{proof}
If a biquasigroup $(Q,\circ,*)$ linear over a group $(Q,+)$ satisfies \eqref{e5}, then $\varphi a+a+\psi b=a$ and $\varphi=\varepsilon$. Thus $a+\psi b=0$. This together with \eqref{e5} for $y=0$ gives $\psi z+\psi\beta z=0$. So, $\beta z=-z$ for all $z\in Q$. Thus $(Q,+)$ is a commutative group. Consequently, $\alpha=\varepsilon$. Therefore, $x\circ y=x+\psi y-\psi b$, \ $x*y=x-y+b$.

The proof of the converse follows from a direct calculation and is omitted.
\end{proof}

A biquasigroup $(Q,\circ,*)$ linear over a group and satisfying \eqref{e5} is medial and both its quasigroups $(Q,\circ)$ and $(Q,*)$ have the same right neutral element. If $\psi^2=\varepsilon$ then this biquasigroup is also  paramedial.

%%%%%%%%%%%%%%%%%%%%%%%%%%%%%%%%%%%%%%%%%%%%%%%%%%%%%%%%%%%%%%%%%%%%%%%%%%%%%%%%%%%%%%%%%%%%%%%%%%%%%%%%
\bigskip\noindent {\bf 6.}
Will now consider a biquasigroup $(Q,\circ,*)$ satisfying the identity \eqref{e6}, i.e.
\begin{eqnarray*}
&&(x\circ z)*(y\circ z)=x\circ y.
\end{eqnarray*}

\begin{theorem}\label{T26} A biquasigroup $(Q,\circ,*)$ satisfies the identity $(\ref{e6})$ if and only if there is a group $(G,\cdot)$ and a bijection $\alpha$ on $Q$ such that $x\circ y=(\alpha x)^{-1}\cdot(\alpha y)$ and $x*y=x\cdot y^{-1}$. 
\end{theorem}
\begin{proof} $\Rightarrow$: Let $x,y,z\in Q$. Then for fixed $q\in Q$ there are $x',y',z'\in Q$ such that and $x=x'\circ q$, $y=y'\circ q$ and $z=z'\circ q$. Then, $x*z=(x'\circ q)*(z'\circ q)=x'\circ z'$ and $y*z=(y'\circ q)*(z'\circ q)=y'\circ z'$. So, 
$$(x*z)*(y*z)=(x'\circ z')*(y'\circ z')=x'\circ y'=(x'\circ q)*(y'\circ q)=x*y.$$ Therefore, $(Q,*)$ is a Ward quasigroup and there exists a group $(Q,\cdot)$ such that $x*y=x\cdot y^{-1}$ and $x\cdot y=x*(e*y)$, where $e=w*w$ for any $w\in Q$ and $x^{-1}=e*x$.

Let $x\in Q$. Then, $e=(x\circ z)*(x\circ z)=x\circ x$. So, $e*(x\circ y)=(y\circ y)*(x\circ y)=y\circ x$, for any $y\in Q$. Let $\alpha x=e\circ x$. Then, $(\alpha x)^{-1}=e*(e\circ x)=x\circ e$. Thus, $(\alpha x)^{-1}\cdot (\alpha y)=(x\circ e)\cdot (e\circ y)=(x\circ e)*(e*(e\circ y))=(x\circ e)*(y\circ e)=x\circ y$.
  	
$\Leftarrow$: Let $x,y,z\in Q$. Then, $(x\circ z)*(y\circ z)=[(\alpha x)^{-1}\cdot (\alpha z)]*[(\alpha y)^{-1}\!\cdot (\alpha z)]=(\alpha x)^{-1}\!\cdot(\alpha z)\cdot(\alpha z)^{-1}\!\cdot(\alpha y)=(\alpha x)^{-1}\!\cdot(\alpha y) =x\circ y$.    
\end{proof}
\begin{corollary}\label{C-uni}
If a biquasigroup $(Q,\circ,*)$ satisfies the identity $(\ref{e6})$, then it is unipotent.
\end{corollary}

\begin{corollary}
If in a biquasigroup $(Q,\circ,*)$ satisfying the identity $(\ref{e6})$ one of quasigroups $(Q,\circ)$ and $(Q,*)$ is commutative, then also the second is commutative. In this case both quasigroups are induced by the same Boolean group.
\end{corollary}

\begin{proposition}
A biquasigroup $(Q,\circ,*)$ linear over a group $(Q,+)$ satisfies the identity $\eqref{e6}$ if and only if a group $(Q,+)$ is commutative, $ x\circ y=\varphi x-\varphi y+a$ and  $x*y=x-y+a$.
\end{proposition}
\begin{proof}
If a biquasigroup $(Q,\circ,*)$ linear over a group $(Q,+)$ satisfies the identity $\eqref{e6}$ then 
$\alpha a+b+\beta a=a$ and $\alpha=\varepsilon$. So, $b+\beta a=0$. Thus \eqref{e6} for $x=y=0$ gives $\psi z+\beta\psi z=0$ which means that $\beta v=-v$ for each $v\in Q$. Hence $(Q,+)$ is commutative and $a=b$.
Therefore  $x*y=x-y+a$.  Substituting this operation to \eqref{e6} we obtain 
$x\circ y=\varphi x-\varphi y+a$.

The converse statement is obvious.
\end{proof}
\begin{corollary} 
A linear biquasigroup satisfying the identity \eqref{e6} is unipotent.
\end{corollary}

%%%%%%%%%%%%%%%%%%%%%%%%%%%%%%%%%%%%%%%%%   
   \medskip\noindent {\bf 7.}
Will now consider a biquasigroup $(Q,\circ,*)$ satisfying the identity \eqref{e7}, i.e.
$$
(x\circ z)\circ (y* z)=x* y.
$$
            
A simple example of a biquasigroup $(Q,\circ,*)$ satisfying the identity \eqref{e7} is a commutative group $(Q,+)$ with the operations $x\circ y=y-x$ and $x*y=x+y$. This biquasigroup is medial, both quasigroups $(Q,\circ)$ and $(Q,*)$ have left neutral element but only the first is unipotent.

Suppose now that in a biquasigroup $(Q,\circ,*)$ satisfying the identity \eqref{e7} the first quasigroup is medial and the second is idempotent. Then, by Toyoda Theorem (cf.\ \cite{Scerb}), there exists a commutative group $(Q,+)$ and its commuting automorphisms $\varphi,\psi$ such that $x\circ y=\varphi x+\psi y+a$ for some fixed $a\in Q$. Then $x*y=(x\circ y)\circ (y*y)=(x\circ y)\circ y=\varphi^2 x+\varphi\psi y+\psi y+\varphi a+a$. This, by \eqref{e7}, implies $\varphi^2-\varphi=\varepsilon$, $\varphi+\varphi\psi+\psi=0$ and $\varphi a=-a$. Thus $a=0$ and $x*y=\varphi^2 x+\varphi\psi y+\psi y$. Since $(Q,*)$ is idempotent, $\varphi^2+\varphi\psi+\psi=\varepsilon$. Hence $x*y=\varphi^2 x+y-\varphi^2 y=\varphi^2 x-\varphi y$. Consequently $(Q,*)$ is medial. Therefore $(Q,\circ,*)$ is medial too.

\medskip
In this way we have proved
\begin{proposition}
If in a biquasigroup $(Q,\circ,*)$ satisfying the identity \eqref{e7} the first quasigroup is medial and the second is idempotent, then the second is medial too and there exists a commutative group $(Q,+)$ and its commuting automorphisms $\varphi,\psi$ such that $\varphi+\varphi\psi+\psi=0$, $\varphi^2=\varphi+\varepsilon$, $x\circ y=\varphi x+\psi y$ and  $x*y=\varphi^2 x-\varphi y$.
\end{proposition}

Conversely we have:
\begin{proposition}\label{P71} Let $(Q,+)$ be a commutative group and $\varphi,\psi$ be its commuting automorphisms such that $\varphi+\varphi\psi+\psi=0$ and $\varphi^2=\varphi+\varepsilon$. Then $(Q,\circ,*)$, where $x\circ y=\varphi x+\psi y$ and $x*y=\varphi^2 x-\varphi y$, is a medial biquasigroup satisfying the identity $\eqref{e7}$. 
\end{proposition}
\begin{proof} 
This is a straightforward calculation.  
\end{proof}

As a consequence of the above results we obtain

\begin{corollary}\label{C70} An idempotent medial biquasigroup $(Q,\circ,*)$ satisfies the identity $\eqref{e7}$ if and only if there exist a commutative group $(Q,+)$ and its automorphism $\varphi$ such that $x\circ y=\varphi x+y-\varphi y$, $x*y=\varphi^2 x-\varphi y$ and $\varphi^2=\varphi+\varepsilon$. 
\end{corollary}

In the case of quasigroups induced by the group $\mathbb{Z}_n$ we have stronger result. For simplicity the value of the integer $t\geqslant 0$ modulo $n$ will be denoted by $[t]_n$. 

\begin{corollary}\label{C71} An idempotent medial biquasigroup induced by the group $\mathbb{Z}_n$ satisfies the identity $\eqref{e7}$ if and only if has the form $(\mathbb{Z}_n,\circ,*)$, where  $x\circ y=[ax+(1-a)y]_n$, $\,x*y=[a^2x+(1-a^2)y]_n$ and $[a^2-a]_n=1$.
\end{corollary}

\begin{corollary}\label{C72} For every $a\geqslant 3$ there is an idempotent medial biquasigroup of order $n=a^2-a-1$ satisfying $\eqref{e7}$. It has the form $(\mathbb{Z}_n,\circ,*)$, where $x\circ y=[ax+(1-a)y]_n$ and $x*y=[(a+1)x-ay]_n$, or $x\circ y=[(1-a)x+ay]_n$ and $x*y=[(2-a)x+(a-1)y]_n$.
\end{corollary}

\begin{proposition}
A medial biquasigroup $(Q,\circ,*)$ satisfies the identity $\eqref{e7}$ if and only if there exist a commutative group $(Q,+)$ and its commuting automorphisms $\varphi,\psi$ such that  $x\circ y=\varphi x+\psi y+c$, $x*y=\varphi^2 x-\varphi y+d$, $\varphi\psi+\varepsilon=0$ and $\varphi c+\psi d+c=d$ for some fixed $c,d\in Q$. 
\end{proposition}

\begin{proposition}
A medial biquasigroup $(\mathbb{Z}_n,\circ,*)$ satisfies the identity $\eqref{e7}$ if and only if there exists $a,b,c,d\in\mathbb{Z}_n$ such that $[ab+1]_n=0$, $[ac+bd+c]_n=d$, $x\circ y=[ax+by+c]_n$ and $x*y=[a^2x-ay+d])_n$.
\end{proposition}

For linear biquasigroup we have the following result. 
\begin{theorem}\label{T27}
A biquasigroup $(Q,\circ,*)$ linear over a group $(Q,+)$ satisfies the identity \eqref{e7} if and only if $(Q,+)$ is a commutative group,  
$x\circ y=\varphi x+\psi y+a$, \ $x*y=\varphi^2 x +\psi\varphi^2 y+b,$ \ $\varphi\psi+\psi^2\varphi^2=0$ and $\varphi a+a+\psi b=b$. 
\end{theorem}
\begin{proof}
If a biquasigroup $(Q,\circ,*)$ linear over a group $(Q,+)$ satisfies \eqref{e7}, then  
$\varphi a+a+\psi b=b$ and $\varphi^2=\alpha$. Thus \eqref{e7} can be reduced to     
$$
\varphi a+\varphi\psi z+a+\psi\alpha y+\psi b+\psi\beta z=b+\beta y,
$$
which for $z=0$ gives $\varphi a+a+\psi\alpha y+\psi b=b+\beta y=(\varphi a+a+\psi b)+\beta y$. Hence
$\psi\alpha y+\psi b=\psi b+\beta y$. So $\psi\alpha y=\psi b+\beta y -\psi b$. Therefore the previous identity implies $\varphi\psi z+a+\psi b+\beta y +\psi\beta z=a+\psi b+\beta y$. Since every element $v\in Q$ can be presented in the form $v=a+\psi b+\beta y$, the last identity means that $\varphi\psi z+v+\psi\beta z=v$ for all $v,z\in Q$. This implies $\varphi\psi=-\psi\beta$. Hence $\varphi\psi z+v=v-\psi\beta z=v+\varphi\psi z$. So, $(Q,+)$ is commutative. Applying these facts to \eqref{e7} we can see that $\beta=\psi\alpha$. Hence $x\circ y=\varphi x+\psi y+a$ and $x*y=\varphi^2 x +\psi\varphi^2 y+b.$

The proof of the converse follows from a direct calculation and is omitted.
\end{proof}

%%%%%%%%%%%%%%%%%%%%%%%%%%%%%%%%%%%%%%%%%%%%%%%%%%%%%%%%%%%%%%%
\bigskip\noindent
{\bf 8.} Will now consider a biquasigroup $(Q,\circ,*)$ satisfying the identity \eqref{e8}, i.e.
$$
(x\circ z)* (y* z)=x\circ y.
$$
\begin{proposition}
If a biquasigroup $(Q,\circ,*)$ satisfies \eqref{e8}, then $(Q,*)$ has no more than one idempotent. If such idempotent exists then it is a right neutral element of a quasigroup $(Q,*)$. 
Moreover, if $x\circ x=u$ for some $u\in Q$ and all $x\in Q$, then %$(Q,*)$ is unipotent, 
$x\circ y=x*(y*x)$, $x*x=w$ and $w\circ u=w$ for all $x,y\in Q$.
\end{proposition}
\begin{proof}
Let $e*e=e$. Since $(Q,\circ)$ is a quasigroup, each $z\in Q$ can be expressed in the form $z=x\circ e$. Thus $z=x\circ e=(x\circ e)*(e*e)=z*e$, so $e$ is a right neutral element of $(Q,*)$. 
If $\bar{e}$ is the second idempotent of $(Q,*)$, then $\bar{e}*e=\bar{e}=\bar{e}*\bar{e}$. Therefore $\bar{e}=e$, so $(Q,*)$ has no more than one idempotent. If $x\circ x=u$ for all $x\in Q$, then, by \eqref{e8}, $u*(x*x)=(x\circ x)*(x*x)=x\circ x=u$. Analogously $u*(y*y)=u$. Thus, $x*x=y*y=w$  for some $w\in Q$, i.e. $(Q,*)$ is unipotent and $w$ is its right neutral element. Then, $x\circ y=(x\circ x)*(y*x)=u*(y*x)$ and $w\circ u=(w\circ w)*(u*w)=u*u=w$. 
\end{proof}

Let $(Q,\circ,*)$ be linear over a group $(Q,+)$.  If it satisfies \eqref{e8}, then 
$\alpha a+b+\beta b=a$, \ $\alpha=\varepsilon$ and $\beta b=-b$. Thus \eqref{e8} can be reduced to
\begin{equation}\label{e81}
\psi z+b+\beta y+\beta b+\beta^2 z=\psi y.
\end{equation}
This for $y=0$ gives $\psi z+\beta^2 z=0$. So, $\psi=-\beta^2$ and $\psi b=-b=\beta b$. 

Now putting  $y=b$ in \eqref{e81} we obtain $\psi z+b+\beta b+\beta b+\beta^2 z=\psi b$, i.e. $-\beta^2 z+\beta b+\beta^2 z=\psi b=\beta b$. So, $\beta b+\beta^2 z=\beta^2 z+\beta b$. This means that $b$ is in the center of $(Q,+)$.  Thus putting $z=0$ in \eqref{e81} and using the above facts, we obtain  $\beta=\psi=-\beta^2$. Hence $\beta=-\varepsilon$. So $(Q,+)$ is commutative, $x\circ y=\varphi x-y+a$ and $x*y=x-y+b$.

In this way, we have proved the “only if” part of Theorem \ref{T28} below. The proof of the converse part of Theorem \ref{T28} follows from a direct calculation and is omitted.  
\begin{theorem}\label{T28}
A biquasigroup $(Q,\circ,*)$ linear over a group $(Q,+)$ satisfies \eqref{e8} if and only if $(Q,+)$ is commutative, \ $x\circ y=\varphi x-y+a$ and $x*y=x-y+b$.
\end{theorem}
\begin{corollary}
A linear biquasigroup satisfying \eqref{e8} is medial.
\end{corollary}

\begin{corollary}
A medial biquasigroup induced by the group $\mathbb{Z}_n$ satisfies $\eqref{e8}$ if and only if 
$x\circ y=[ax-y+c]_n$ and $x*y=[x-y+d]_n$ for some $a,c,d\in\mathbb{Z}_n$ such that $(a,n)=1$.
\end{corollary}            

%%%%%%%%%%%%%%%%%%%%%%%%%%%%%%%%%%%%%%%%%%%%%%%
\bigskip\noindent
{\bf 9.} Finally, let us consider a biquasigroup $(Q,\circ,*)$ satisfying the identity \eqref{e9}, i.e.
$$
(x\circ z)* (y* z)=x*y.
$$     

\begin{theorem}\label{T29} In a biquasigroup $(Q,\circ,*)$ satisfying the identity \eqref{e9} the quasigroups $(Q,\circ)$ and $(Q,*)$ have no more than one idempotent. If such idempotent exists then it is a common right neutral element of these quasigroups.
\end{theorem}
\begin{proof}
Assume $(Q,\circ)$ has an idempotent $a$. Then $a*a=(a\circ a)*(a*a)=a*(a*a)$ and so $a*a=a$. Analogously, for $a*a=a$ we have $a*a=(a\circ a)*(a*a)=(a\circ a)*a$, which implies $a\circ a=a$. So, $(Q,\circ)$ and $(Q,*)$ have the same idempotent. Then for each $x\in Q$ $x*a=(x\circ a)*(a*a)=(x\circ a)*a$, which implies $x=x\circ a$. Thus $a$ is a right neutral element of $(Q,\circ)$. On the other hand, $x\circ a=x$ gives $x*x=(x\circ a)*(x*a)=x*(x*a)$, and consequently $x=x*a$. Thus, $a$ is a right neutral element of $(Q,\circ)$ and $(Q,*)$.
 \end{proof}

\begin{corollary} 
If in a biquasigroup $(Q,\circ,*)$ satisfying $\eqref{e9}$ the quasigroup $(Q,*)$ is unipotent, then $(Q,\circ)=(Q,*)$ and $(Q,\circ)$ is a Ward quasigroup.
\end{corollary}
\begin{proof} 
Let $x*x=a$ for all $x\in Q$ and some $a\in Q$. Then $a*a=x*x=(x\circ x)*(x*x)=(x\circ x)*a$. Therefore, $x\circ x=a$, i.e. $(Q,\circ)$ is unipotent. Consequently, $a*(x*y)= (y\circ y)*(x*y)=y*x$, which implies $x\circ y=(x\circ y)*a=(x\circ y)*(y*y)=x*y$. Hence $(Q,\circ)=(Q,*)$ and \eqref{e9} coincides with \eqref{e1}. This means that $(Q,\circ)$ is a Ward quasigroup.
\end{proof}

\begin{theorem}\label{T29}
A biquasigroup $(Q,\circ,*)$ linear over a group $(Q,+)$ satisfying the identity \eqref{e9} is medial and can be presented in the form $x\circ y=x-\beta^2y-\beta b$ and $x*y=x+\beta y+b$, where $(Q,+)$ is a commutative group, $\beta\in {\rm Aut}(Q,+)$ and $b\in Q$. This biquasigroup has a right neutral element $e=-\varphi^{-1}b$. 
 
Conversely, if $(Q,+)$ is a commutative group, $\beta\in {\rm Aut}(Q,+)$, $b\in Q$, $x\circ y=x-\beta^2y-\beta b$ and $x*y=x+\beta y+b$, then the biquasigroup  $(Q,\circ,*)$ satisfies \eqref{e9}.
\end{theorem}
\begin{proof}
If a biquasigroup $(Q,\circ,*)$ linear over a group $(Q,+)$ satisfies the identity \eqref{e9}, then 
$\alpha a+b+\beta b=b$ and $\varphi=\varepsilon$. So, \eqref{e9} can be reduced to 
\begin{equation}\label{e91}
\alpha a+\alpha\psi z+b+\beta\alpha y+\beta b+\beta^2 z=b+\beta y,
\end{equation} 
which for $z=0$ gives $\alpha a+b+\beta\alpha y+\beta b=b+\beta y=\alpha a+b+\beta b+\beta y$. Since $\alpha a=b=b-\beta b$, the last implies $\beta\alpha y=\beta b+\beta y-\beta b$.
This together with \eqref{e91} (for $y=v$) gives
\begin{equation}\label{e92}
\alpha a+\alpha\psi z+b+\beta b+\beta v+\beta^2 z=b+\beta v.
\end{equation}
 Now adding $\beta v$ on the right side to \eqref{e91} and putting $y=0$ we get
$$
\alpha a+\alpha\psi z+b+\beta b+\beta^2 z+\beta v=b+\beta v.
$$
Comparing this identity with \eqref{e92} we obtain $\beta v+\beta^2z=\beta^2z+\beta v$ for all $v,z\in Q$.
This shows that $(Q,+)$ is a commutative group.  Consequently, $\beta\alpha y=\beta y$, so $\alpha=\varepsilon$. This by $\alpha a+b+\beta b=b$ gives $a=-\beta b$. Again putting $y=0$ in \eqref{e91} and using the above facts we obtain $\psi=-\beta^2$. Therefore, $x\circ y=x-\beta^2y-\beta b$ and $x*y=x+\beta y+b$.

The proof of the converse part of the Theorem follows from a direct calculation and is omitted.
\end{proof}
\begin{proposition}
A medial biquasigroup $(\mathbb{Z}_n,\circ,*)$ satisfies the identity $\eqref{e9}$ if and only if $x\circ y=[x-a^2y-ab]_n$, $x*y=[x+ay+b]_n$, where $a,b\in\mathbb{Z}_n$ are fixed and $(a,n)=1$.
\end{proposition}

\begin{example}\rm
Let $n=a^2+1>4$. Then $(\mathbb{Z}_n,\circ,*)$, where $x\circ y=[x+y]_n$ and $x*y=[x+ay]_n$ is an example of a biquasigroup satisfying \eqref{e9}.
\end{example}

\medskip\noindent
{\bf 10.} Many authors study linear quasigroups of the second type, namely quasigroups $(Q,\cdot)$ where, in the definition of the operation, the constant element is not placed in the middle of the formula but at its end, i.e. $x\cdot y=\varphi x+\psi y+a$.

Biquasigroups of this type satisfying the identities $(2)-(9)$ coincide with the quasigroups of the previous type. Namely, if a biquasigroup 
$\widehat{Q}=(Q,\circ,*)$ with the operations $x\circ y=\varphi x+\psi y+a$ and $x*y=\alpha x+\beta y+b$, where $\alpha,\beta, \varphi,\psi$ are automorphisms of a group $(Q,+)$,
satisfies $(2)$ then $\alpha a+\beta a=0$ and $\varphi=\varepsilon$. Thus $\alpha\psi z+\alpha a+\beta y+\beta\psi z+\beta a=\beta y$. This for $y=0$ and $\psi z=v$ gives $\alpha v=-\beta a-\beta v-\alpha a$. Since $\alpha$ and $\beta$ are automorphisms of $(Q,+)$ the last expression for $v=u+w$ implies $\beta(u+w)=\beta w+\beta u$. Thus $\beta u+\beta w=\beta u+\beta w$ for all $u,w\in Q$. Hence $(Q,+)$ is a commutative group. Such biquasigroups are described in subsection {\bf 2}. % Th. 2.

If a biquasigroup $(Q,\circ,\circ)$ with $x\circ y=\varphi x+\psi y +a$ satisfies \eqref{e3}, then $\varphi a+\psi a=0$, $\varphi=\varepsilon$ and $\psi z+a+\psi y+\psi^2z+\psi a=\psi y$, which for $y=a$ gives $\psi=-\varepsilon$. This shows that $(Q,+)$ is a commutative group and $x\circ y=x-y+a$. Also in the case when $(Q,\circ)$ with $x\circ y=\varphi x+a+\psi y$ satisfies \eqref{e1}, the group $(Q,+)$ must be commutative and $x\circ y=x-y+a$. This means that these two cases coincide.

If a biquasigroup $\widehat{Q}$ satisfies \eqref{e4}, then $\varphi a+\psi a+a=b$ and $\alpha=\varphi^2$. Because by Theorem \ref{T24} we have $q=\varphi 0+\psi 0+a=\alpha 0+\beta 0+b$, must be $q=a=b$. Consequently, $a=\varphi x+\psi x+a$. This implies $\varphi=-\psi$, which together with $(x\circ 0)\circ (x\circ 0)=a$ implies $\varphi^2x+\varphi a-\varphi^2 x-\varphi a=0$. Hence $\varphi x+a=a+\varphi x$ for all $x\in Q$. So, $a$ is in the center of $(Q,+)$. Therefore this case is reduced to the case described in subsection {\bf 4}. %Theorem \ref{T24a}.

If a biquasigroup $\widehat{Q}$ satisfies \eqref{e5}, then  by Theorem \ref{T25} the quasigroup $(Q,*)$ has a right neutral element $e$. Thus $x=x*e=\alpha x+\beta e+b$ for all $x\in Q$. In particular $0=0*e=\beta e+b$. Consequently, $x=x*e=\alpha x$ and $x*y=x+\beta y+b$. Applying this formula to \eqref{e5} we can see that $\varphi=\varepsilon$ and $\varphi b=-a$. Therefore the identity \eqref{e5} can be written in the form
$$
\psi z+a+\psi y+\psi\beta z=\psi y+a.
$$
This for $z=0$ implies $a+\psi y=\psi y+a$. Hence $a$ is in the center of $(Q,+)$. Also $b$ is in the center of $(Q,+)$ because $\varphi b=-a$. Thus this case reduces to the case from subsection {\bf 5}.

By Corollary \ref{C-uni} any quasigroup satisfying \eqref{e6} is unipotent. Thus if $\widehat{Q}$ satisfies \eqref{e6}, then 
$\alpha 0+\beta 0+b=b$ implies $b=x*x=\alpha x+\beta x+b$, i.e. $\beta x=-\alpha x$  for all $x\in Q$. From \eqref{e6} it follows $\alpha=\varepsilon$. Thus $\beta x=-x$. Since $\beta$  is an automorphism of $(Q,+)$, $(Q,+)$ is commutative. Hence this case reduces to subsection {\bf 6}.

If a biquasigroup $\widehat{Q}$ satisfies \eqref{e7}, then $\varphi a+\psi b+a=b$ which together with \eqref{e7} fort $x=y=0$ implies
$$
\varphi\psi z+\varphi a+\psi\beta z+\psi b+a=b=\varphi a+\psi b+a.
$$
Thus $\varphi\psi z=\varphi a-\psi\beta z-\varphi a$. Since $\varphi\psi$ and $\psi\beta$ are automorphisms of $(Q,+)$ the last for $z=u+v$ gives 

\medskip
$\varphi\psi(u+v)=\varphi a-\psi\beta (u+v)-\varphi a$.

\medskip\noindent
On the other side,

\medskip
$\varphi\psi u+\varphi\psi v=\varphi a-\psi\beta u-\varphi a+\varphi a-\psi\beta v-\varphi a=\varphi a-\psi\beta (v+u)-\varphi a$.

\medskip\noindent
Comparing these two expression we obtain
$\psi\beta(u+v)=\psi\beta(v+u)$. Hence $(Q,+)$ is a commutative group and this case reduces to {\bf 7}.

If a biquasigroup $\widehat{Q}$ satisfies \eqref{e8}, then $\alpha a+\beta b+b=a$ and $\alpha=\varepsilon$. This together with \eqref{e8} for $x=y=0$ implies $\psi z+a+\beta^2z=a$. Hence $\beta^2z=-a-\psi z+a$. From this for $z=u+v$, in a similar way as in the previous case, we obtain $\psi(u+v)=\psi(v+u)$. Therefore $(Q,+)$ is a commutative group and this case reduces to {\bf 8}.

The case when $\widehat{Q}$ satisfies \eqref{e9} reduces to {\bf 9}. Indeed, in this case $\alpha a+\beta b=0$, which together \eqref{e9} for $x=y=0$ shows that $\beta^2z=-\alpha a-\alpha\psi z-\beta b$. From this we compute $\alpha\psi(u+v)=\alpha\psi(v+u)$. Hence $(Q,+)$ is commutative.

\footnotesize{\rightline{Received \ February 03, 2021}
\noindent
W.A. Dudek \\
 Faculty of Pure and Applied Mathematics,
 Wroclaw University of Science and Technology\\
 50-370 Wroclaw,  Poland \\
 Email: wieslaw.dudek@pwr.edu.pl\\[4pt]
R.A.R. Monzo\\
Flat 10, Albert Mansions, Crouch Hill, London N8 9RE, United Kingdom\\
E-mail: bobmonzo@talktalk.net}

\end{document}